\documentclass[oneside,reqno,english]{amsart}
\usepackage[T1]{fontenc}
\usepackage[latin9]{inputenc}
\usepackage{color}
\usepackage{babel}
\usepackage{prettyref}
\usepackage{mathrsfs}
\usepackage{mathtools}
\usepackage{amstext}
\usepackage{amsthm}
\usepackage{amssymb}
\usepackage{graphicx}
\usepackage{setspace}
\usepackage[all]{xy}
\usepackage[pdfusetitle,
 bookmarks=true,bookmarksnumbered=false,bookmarksopen=false,
 breaklinks=false,pdfborder={0 0 0},pdfborderstyle={},backref=false,colorlinks=false]
 {hyperref}
\hypersetup{
 colorlinks=true,citecolor=blue,linkcolor=blue,linktocpage=true}

\makeatletter
\numberwithin{equation}{section}
\numberwithin{figure}{section}
\theoremstyle{plain}
\newtheorem{thm}{\protect\theoremname}[section]
\theoremstyle{definition}
\newtheorem{defn}[thm]{\protect\definitionname}
\theoremstyle{plain}
\newtheorem{lem}[thm]{\protect\lemmaname}
\theoremstyle{plain}
\newtheorem{prop}[thm]{\protect\propositionname}
\theoremstyle{definition}
\newtheorem{example}[thm]{\protect\examplename}
\theoremstyle{remark}
\newtheorem{rem}[thm]{\protect\remarkname}
\theoremstyle{plain}
\newtheorem{cor}[thm]{\protect\corollaryname}

\@ifundefined{date}{}{\date{}}
\newrefformat{cor}{Corollary~\ref{#1}}
\newrefformat{subsec}{Section~\ref{#1}}
\newrefformat{lem}{Lemma~\ref{#1}}
\newrefformat{thm}{Theorem~\ref{#1}}
\newrefformat{sec}{Section~\ref{#1}}
\newrefformat{chap}{Chapter~\ref{#1}}
\newrefformat{prop}{Proposition~\ref{#1}}
\newrefformat{exa}{Example~\ref{#1}}
\newrefformat{tab}{Table~\ref{#1}}
\newrefformat{rem}{Remark~\ref{#1}}
\newrefformat{def}{Definition~\ref{#1}}
\newrefformat{fig}{Figure~\ref{#1}}
\newrefformat{claim}{Claim~\ref{#1}}

\makeatother

\providecommand{\corollaryname}{Corollary}
\providecommand{\definitionname}{Definition}
\providecommand{\examplename}{Example}
\providecommand{\lemmaname}{Lemma}
\providecommand{\propositionname}{Proposition}
\providecommand{\remarkname}{Remark}
\providecommand{\theoremname}{Theorem}

\begin{document}
\title[]{New duality in choices of feature spaces via kernel analysis}
\author{Palle E.T. Jorgensen}
\address{(Palle E.T. Jorgensen) Department of Mathematics, The University of
Iowa, Iowa City, IA 52242-1419, U.S.A.}
\email{palle-jorgensen@uiowa.edu}
\author{James Tian}
\address{(James F. Tian) Mathematical Reviews, 416 4th Street Ann Arbor, MI
48103-4816, U.S.A.}
\email{james.ftian@gmail.com}
\begin{abstract}
We present a systematic study of the family of positive definite (p.d.)
kernels with the use of their associated feature maps and feature
spaces. For a fixed set $X$, generalizing Loewner, we make precise
the corresponding partially ordered set $Pos\left(X\right)$ of all
p.d. kernels on $X$, as well as a study of its global properties.
This new analysis includes both results dealing with applications
and concrete examples, including such general notions for $Pos\left(X\right)$
as the structure of its partial order, its products, sums, and limits;
as well as their Hilbert space-theoretic counterparts. For this purpose,
we introduce a new duality for feature spaces, feature selections,
and feature mappings. For our analysis, we further introduce a general
notion of dual pairs of p.d. kernels. Three special classes of kernels
are studied in detail: (a) the case when the reproducing kernel Hilbert
spaces (RKHSs) may be chosen as Hilbert spaces of analytic functions,
(b) when they are realized in spaces of Schwartz-distributions, and
(c) arise as fractal limits. We further prove inverse theorems in
which we derive results for the analysis of $Pos\left(X\right)$ from
the operator theory of specified counterpart-feature spaces. We present
constructions of new p.d. kernels in two ways: (i) as limits of monotone
families in $Pos\left(X\right)$, and (ii) as p.d. kernels which model
fractal limits, i.e., are invariant with respect to certain iterated
function systems (IFS)-transformations.
\end{abstract}

\maketitle

\section{Introduction}

The purpose of this paper is to introduce a new duality for the study
of feature spaces, feature selections, and feature mappings, which
arise in diverse applications of kernel analysis of non-linear problems.
Here, our use of the notion of \textquotedblleft feature space\textquotedblright{}
is in the sense of data science: it refers to the collections of features
used to characterize the data at hand. By \textquotedblleft feature
selection,\textquotedblright{} we mean one or more techniques from
machine learning, typically involving the choice of subsets of relevant
features from the original set to enhance model performance. The term
\textquotedblleft feature mappings\textquotedblright{} refers to a
technique in data analysis and machine learning for transforming input
data from a lower-dimensional space to a higher-dimensional space
using kernels, enabling easier analysis or classification. Choices
of feature mapping involve constructions and optimization algorithms,
which lead to the selection of specific functions. These mappings
serve to transform the original data into a new set of features (feature
spaces) that better capture the significant patterns in the data.

We consider families of positive definite (p.d.) kernels, defined
on a product set $X\times X$, where $X$ is merely a set with no
extra \emph{a priori} structure. The positivity condition for $K$
(p.d.), \prettyref{def:b1}, was first studied by Aronszajn \cite{MR51437}
and his contemporaries (see also \cite{MR3526117,MR1161974,MR4601874,MR4569938}). 

Pairs $\left(X,K\right)$ arise in various contexts, including optimization,
principal component analysis (PCA), partial differential equations
(PDEs), and statistical inference. In stochastic models, the p.d.
kernel often serves as a covariance kernel of a Gaussian field. We
emphasize that the set $X$ does not lend itself to direct analysis;
in particular, it does not come equipped with any linear structure.
However, measurements performed on $X$ often lead to p.d. kernels
$K$. Moreover, $K$ then allows one to represent data from $X$ in
a linear space, often referred to as a feature space. 

We say that a pair $\left(\phi,\mathscr{H}\right)$ represents a feature
map, and a feature space if $\phi$ is a function from $X$ mapping
into the Hilbert space $\mathscr{H}$ in such a way that $K$ is recovered
from the inner product in $\mathscr{H}$ via $\phi$; see \prettyref{prop:c3}
below. There is a vast variety of choices of feature selections in
the form $\left(\phi,\mathscr{H}\right)$, and Aronszajn's reproducing
kernel Hilbert space (RKHS), denoted as $\mathscr{H}_{K}$, is only
one possibility. 

In this paper, we introduce a new duality approach to the study of
choices of feature selections, and apply it to particular p.d. kernels
$\left(X,K\right)$ arising in both pure and applied mathematics.
For related papers on feature selection, we refer to \cite{MR4567343,MR4561157,MR4536608,MR4467152,MR4295177,MR4302453}
as well as the references cited therein.

\textbf{Organization.} Our duality tools are outlined in detail below,
and in more detail, in \prettyref{subsec:3-1}, especially Propositions
\ref{prop:3-5} and \ref{prop:3-7}. \prettyref{sec:4} deals with
the need for choices of \textquotedblleft bigger\textquotedblright{}
spaces for implementation, which includes here choices of Hilbert
spaces of Schwartz distributions, i.e., generalized functions. For
optimization questions arising in practice, it is important to have
useful ordering of families of kernels, as well as monotone limit
theorems for kernels, and these two questions are addressed systematically
in \prettyref{sec:F}. Applications of our kernel duality principles
and new transforms, are addressed in \prettyref{sec:G}. 

\textbf{Applications.} While our focus is primarily theoretical, kernel
theory and optimization have had significant impact on practical applications,
particularly in machine learning algorithms and big data analysis
\cite{MR2810909,MR3173967,MR3024752,MR2597189,MR2579912,MR4300781,MR3969356}.
Beyond these areas, kernel methods have found relevance in fields
such as statistical inference, quantum dynamics, perturbation theory,
and operator algebras, with applications ranging from multiplicative
change-of-measure algorithms to the analysis of coherent states and
Fock spaces \cite{MR4405120,MR4241109,MR4113042,MR4059341,MR3994389,MR3439677}.
This versatility has sparked renewed interest in both the theoretical
foundations and practical applications of kernels, leading to deeper
understanding of inference techniques and optimization models \cite{MR4655783,MR3875593,MR1482832,MR4751246,MR4735224,MR4734568,MR4689655,MR4635411}.

Our approach focuses on identifying duality principles to better understand
feature spaces and their role in kernel methods. While the ``kernel
trick'' in machine learning often bypasses explicit constructions
of the ambient feature space, studying its structure offers a richer
theoretical perspective. Insights into feature space flexibility,
stability, and scalability can inform the design and optimization
of kernels, enhancing algorithmic robustness and interpretability.
This dual perspective bridges practical applications like clustering,
support-vector machines, and principal component analysis with the
broader mathematical framework, enabling more sophisticated kernel-based
solutions for complex data problems.

\section{\label{sec:2}Preliminaries}

In this section we introduce the main notions which will be used inside
the paper.

The concept of a kernel in machine learning is powerful tool used
in in the design of Support Vector Machines (SVMs). A kernel is a
function that operates on points from the input space, commonly referred
to as the $X$ space. The primary role of this function is to return
a scalar value, but a Hilbert space, called a reproducing kernel Hilbert
space (RKHS). This higher-dimensional space, known as the $Z$ space.
It conveys how close or similar vectors are in the $Z$ space The
kernel allows one to glean the necessary information about the vectors
in this more complex space without having to access the space directly.
This approach allows one to understand the relationship and position
of vectors in a higher-dimensional space and is a powerful tool for
classification tasks.

In this section, we introduce fundamental definitions, along with
selected lemmas and properties that serve as key building blocks for
the paper.
\begin{defn}[Positive definite]
\label{def:b1}Let $X$ be a set. 
\begin{enumerate}
\item A function $K:X\times X\rightarrow\mathbb{C}$ is said to be a positive
definite (p.d.) kernel if, for all $n\in\mathbb{N}$, all $\left(x_{i}\right)_{i=1}^{n}$
in $X$, and all $\left(c_{i}\right)_{i=1}^{n}$ in $\mathbb{C}$,
we have 
\[
\sum_{i=1}^{n}\sum_{j=1}^{n}\overline{c_{i}}c_{j}K\left(x_{i},x_{j}\right)\geq0.
\]
\item Given a p.d. kernel $K:X\times X\rightarrow\mathbb{C}$, let $\mathscr{H}_{K}$
be the Hilbert completion of the set $H_{0}=span\left\{ K_{x}:x\in X\right\} $,
where $K_{x}\left(\cdot\right)=K\left(\cdot,x\right)$, $x\in X$,
with respect to the norm 
\[
\left\Vert \sum_{i=1}^{n}c_{i}K_{x_{i}}\right\Vert _{\mathscr{H}_{K}}^{2}=\sum_{i=1}^{n}\sum_{j=1}^{n}\overline{c_{i}}c_{j}K\left(x_{i},x_{j}\right).
\]
$\mathscr{H}_{K}$ is called the reproducing kernel Hilbert space
(RKHS) of $K$, and it has the reproducing property:
\[
f\left(x\right)=\left\langle K_{x},f\right\rangle _{\mathscr{H}_{K}}
\]
valid for all $f\in\mathscr{H}_{K}$ and $x\in X$. 
\end{enumerate}
\end{defn}

Throughout this paper, all the Hilbert spaces are assumed separable. 
\begin{lem}[Parseval frame]
\label{lem:b2}Let $\mathscr{H}$ be a Hilbert space, and $\left\{ f_{n}\right\} _{n\in\mathbb{N}}\subset\mathscr{H}$.
Suppose 
\begin{equation}
\sum_{n\in\mathbb{N}}\left|\left\langle f_{n},h\right\rangle _{\mathscr{H}}\right|^{2}=\left\Vert h\right\Vert _{\mathscr{H}}^{2},\quad\forall h\in\mathscr{H}.\label{eq:b1}
\end{equation}
Then $\left\{ f_{n}\right\} $ is an orthonormal basis (ONB) if and
only if $\left\Vert f_{n}\right\Vert _{\mathscr{H}}=1$ for all $n\in\mathbb{N}$.
\end{lem}

\begin{proof}
Fix $n_{0}$ and assume \eqref{eq:b1}, then
\[
\left\Vert f_{n_{0}}\right\Vert _{\mathscr{H}_{K}}^{2}=\left\Vert f_{n_{0}}\right\Vert _{\mathscr{H}_{K}}^{4}+\sum_{n\neq n_{0}}\left|\left\langle f_{n},f_{n_{0}}\right\rangle _{\mathscr{H}_{K}}\right|^{2},
\]
and so $\left\langle f_{n},f_{n_{0}}\right\rangle _{\mathscr{H}_{K}}=0$,
for all $n\in\mathbb{N}\backslash\left\{ n_{0}\right\} $, if $\left\Vert f_{n_{0}}\right\Vert _{\mathscr{H}_{K}}=1$.
\end{proof}
\begin{lem}[Kernel representation]
\label{lem:b3}Let $K$ be a p.d. kernel on $X\times X$. 
\begin{enumerate}
\item \label{enu:b1}A system of functions $\left\{ f_{n}\right\} $ on
$X$ is a Parseval frame for $\mathscr{H}_{K}$ if and only if 
\begin{equation}
K\left(x,y\right)=\sum_{n}\overline{f_{n}\left(x\right)}f_{n}\left(y\right),\quad x,y\in X.\label{eq:b2}
\end{equation}
Moreover, when \eqref{eq:b2} holds, then $\left(f_{n}\left(x\right)\right)\in l^{2}$
for all $x\in X$. 
\item \label{enu:b2}Further, if all the $f_{n}$\textquoteright s are distinct,
i.e., each with multiplicity one, then $\left\{ f_{n}\right\} $ is
an ONB. 
\end{enumerate}
\end{lem}

\begin{proof}
For \eqref{enu:b1}, see e.g., \cite{MR3526117}. Note $\sum_{n}\left|f_{n}\left(x\right)\right|^{2}=\left\Vert K_{x}\right\Vert _{\mathscr{H}}^{2}=K\left(x,x\right)<\infty$,
for all $x\in X$. 

Part \eqref{enu:b2} follows from the argument in \prettyref{lem:b2},
where $\mathscr{H}_{K}\simeq l^{2}$ with the isomorphism $f_{n}\mapsto e_{n}$. 
\end{proof}
\begin{prop}[Products of p.d. kernels]
\label{prop:b4}Let $K$ and $L$ be p.d. kernels defined on $X\times X$.
Set $M=KL$ as follows: $M\left(x,y\right)\coloneqq K\left(x,y\right)L\left(x,y\right)$
for $\left(x,y\right)\in X\times X$. Then $M$ is also p.d. on $X\times X$. 
\end{prop}

\begin{proof}
Pick the representation \eqref{eq:b2} for $K$, and consider $N\in\mathbb{N}$,
$c_{i}\in\mathbb{R}$ (or $\mathbb{C}$), $x_{i}\in X$, $1\leq i\leq N$.
Then we get the desired conclusion as follows: 
\begin{multline*}
\sum_{i}\sum_{j}\overline{c_{i}}c_{j}M\left(x_{i},x_{j}\right)\\
=\sum_{n}\left(\sum_{i}\sum_{j}\left(\overline{c_{i}f_{n}\left(x_{i}\right)}\right)\left(c_{j}f_{n}\left(x_{j}\right)\right)L\left(x_{i},x_{j}\right)\right)\geq0,
\end{multline*}
where we used the p.d. property of $L$ in the last step. 
\end{proof}
\begin{defn}[\emph{Loewner order, see e.g., \cite{MR0486556,MR24487}}]
 \label{def:order}For p.d. kernels $K,L$ on $X\times X$, we say
$K\leq L$ if $L-K$ is p.d. 
\end{defn}

\begin{lem}
If $K:X\times X\rightarrow\mathbb{C}$ is p.d. and $K\geq1$, then
$K^{n}\leq K^{n+1}$, $n\in\mathbb{N}$. 
\end{lem}

\begin{proof}
Recall that products of p.d. kernels are p.d. (see \prettyref{prop:b4},
and also \prettyref{prop:3-7}). Therefore, $K^{n+1}-K^{n}=K^{n}\left(K-1\right)\geq0$. 
\end{proof}
\begin{example}
Let $K\left(z,w\right)=\left(1-\overline{w}z\right)^{-1}$ be the
Szeg\H{o} kernel on $\mathbb{D}\times\mathbb{D}$, where $\mathbb{D}=\left\{ z\in\mathbb{C}:\left|z\right|<1\right\} $.
Then 
\[
\left(1-\overline{w}z\right)^{-1}=\sum_{n=0}^{\infty}\overline{w}^{n}z^{n}\geq1,
\]
in the sense of \prettyref{def:order}, and so 
\[
1\leq K\leq K^{2}\leq\cdots\leq K^{n}.
\]
More specifically, 
\[
1\leq\frac{1}{1-\overline{w}z}\leq\frac{1}{\left(1-\overline{w}z\right)^{2}}\leq\frac{1}{\left(1-\overline{w}z\right)^{3}}\leq\cdots\leq\frac{1}{\left(1-\overline{w}z\right)^{n}}.
\]
Here, $K^{2}\left(z,w\right)=\left(1-\overline{w}z\right)^{-2}$ is
the Bergman kernel. 

Throughout the paper, we will revisit this family of kernels and its
variations, both for motivations and for illustrations of general
results. The reader may refer to Examples \ref{exa:c4}, \ref{exa:c6},
\ref{exa:d2}, \ref{exa:f5}, \ref{exa:f3}, as well as Corollaries
\ref{cor:d3}, \ref{cor:g3}, and \prettyref{rem:e11}.
\end{example}

\begin{prop}[Monotonicity]
Consider two pairs of p.d. kernels $K_{i}$ and $L_{i}$, and form
the product $P_{i}:=K_{i}L_{i}$, $i=1,2$. If $K_{1}\leq K_{2}$,
$L_{1}\leq L_{2}$, then $P_{1}\leq P_{2}$.
\end{prop}

\begin{proof}
By assumption, we get the following conclusion: 
\[
P_{2}-P_{1}=K_{2}L_{2}-K_{1}L_{1}=\left(K_{2}-K_{1}\right)L_{2}+K_{1}\left(L_{2}-L_{1}\right)\geq0.
\]
\end{proof}
Since our paper is interdisciplinary it combines topics from pure
and applied. Of special significance to the above are the following
citations \cite{MR4302453,MR1698952,MR4291375,MR4126821,MR4300781,MR4655783,MR4734568}.

\section{\label{sec:3}Feature space realizations}

The purpose of the present section is to outline key links connection
the notions from \prettyref{sec:2} to the present applications. We
outline the main connections between the key pure math notions (especially
kernel-duality introduced above), and the applied notions, focusing
on feature selection, feature maps, and kernel-machines.

With \textquotedblleft feature selection\textquotedblright{} we refer
to the part of machine learning that identifies the \textquotedblleft best\textquotedblright{}
insights into phenomena/observations. A feature is input, i.e., a
measurable property of the phenomena. In statistical learning, features
are often identified with choices of independent random variables,
typically identically distributed (i.i.d.); see below. More generally,
learning algorithms serve to identify features that yield better models.
Features come in several forms, for example they might be numeric,
or qualitative features. \textquotedblleft Good\textquotedblright{}
feature selections in turn let us identify the important or significant
patterns that distinguish between data forms and instances. Indeed,
as we recall, machine learning is truly multidisciplinary, as is reflected
in for example, how features are viewed. Example: a geometric view,
treating features as tuples, or vectors in a high-dimensional space,
the feature space. Equally important is the probabilistic perspective,
i.e., viewing features as multivariate random variables. The following
references may be helpful, \cite{MR4689655,MR4586941,MR4241109,MR2597189,MR4300781,MR3173967,MR3875593},
and \cite{MR4601874}. An important part of the tools that go into
feature selection in the statistical setting is known as principal
component analysis (PCA) \cite{MR3454256}. It is a dimensionality
reduction of features, i.e., a reduction of the dimensionality of
large data sets. With the use of choices of covariance kernels, it
allows one to transform large sets of variables into smaller ones
that still contains most of the information in the large set; see
e.g., \cite{MR4769269}.

As noted above, in applications such as data analysis, the initial
set $X$ is general and typically unstructured. In particular, in
applications, choices of sets $X$ may not have any linear structure.
But, nonetheless, in the design of optimization models (for example
in statistical inference, and in machine learning models), there will
in fact be natural choices of families of positive definite (p.d.)
kernels, specified on $X\times X$. Each such p.d. kernel will then
yield an RKHS, denoted here $\mathscr{H}_{K}$. And $\mathscr{H}_{K}$
does present one possible choice of feature space (\prettyref{def:C1}),
but applications dictate a qualitative and quantitative comparison
within the variety of feature spaces for a single p.d. kernel $K$.
In particular, we study the possibility of a second p.d. kernel, say
$L$, serving to generate a feature space for $K$ (\prettyref{prop:c3}.)
These themes are addressed below. We further study operations on the
variety of p.d. kernels, as they relate to feature space selection
questions.
\begin{defn}[\emph{Feature map}, and \emph{feature space}]
\label{def:C1}Given a p.d. kernel $K:X\times X\rightarrow\mathbb{C}$
defined on a set $X$, a Hilbert space $\mathscr{L}$ is said to be
a feature space for $K$ if there is a map $\varphi:X\rightarrow\mathscr{L}$,
such that $K\left(x,y\right)=\left\langle \varphi\left(x\right),\varphi\left(y\right)\right\rangle _{\mathscr{L}}$,
for all $x,y\in X$. Set 
\[
H_{S}\left(K\right)\coloneqq\left\{ \left(\varphi,\mathscr{L}\right):K\left(x,y\right)=\left\langle \varphi\left(x\right),\varphi\left(y\right)\right\rangle _{\mathscr{L}},\:x,y\in X\right\} .
\]
\end{defn}

\begin{rem}
$H_{S}\left(K\right)\neq\emptyset$. Some basic examples include:
\begin{enumerate}
\item $\varphi\left(x\right)=K_{x}$, $\mathscr{L}=\mathscr{H}_{K}$ (the
RKHS of $K$), and $K\left(x,y\right)=\left\langle K_{x},K_{y}\right\rangle _{\mathscr{H}_{K}}$.
\item $\varphi\left(x\right)=\delta_{x}$, $\mathscr{L}=\overline{span}\left\{ \delta_{x}:x\in X\right\} $,
where the Hilbert completion is with respect to 
\[
\left\Vert \sum_{i=1}^{N}c_{i}\delta_{x}\right\Vert _{\mathscr{L}}^{2}=\sum_{i,j=1}^{N}\overline{c_{i}}c_{j}K\left(x_{i},x_{j}\right),
\]
and $K\left(x,y\right)=\left\langle \delta_{x},\delta_{y}\right\rangle _{\mathscr{L}}$. 
\item $\varphi\left(x\right)=W_{x}\sim N\left(0,K\left(x,x\right)\right)$,
i.e., $\left(W_{x}\right)_{x\in X}$ is a mean zero Gaussian field,
realized in $\mathscr{L}=L^{2}\left(\Omega,\mathbb{P}\right)$; and
$K\left(x,y\right)=\left\langle W_{x},W_{y}\right\rangle _{L^{2}\left(\mathbb{P}\right)}$. 
\end{enumerate}
More general constructions are considered below.

\end{rem}

\begin{prop}
\label{prop:c3}Given a p.d. kernel $K$ on $X\times X$, then for
every Hilbert space $\mathscr{L}$ with $\dim\mathscr{L}=\dim\mathscr{H}_{K}$,
there exists $\varphi$ such that $\left(\varphi,\mathscr{L}\right)\in H_{S}\left(K\right)$. 
\end{prop}

\begin{proof}
Let $\left\{ f_{n}\right\} $ be an ONB in $\mathscr{H}_{K}$, and
$\left\{ \zeta_{n}\right\} $ an ONB in $\mathscr{L}$. Then the map
\begin{equation}
\varphi\left(x\right)=\sum_{n}f_{n}\left(x\right)\zeta_{n}\in\mathscr{L}\label{eq:s3}
\end{equation}
is well defined since $\left(f_{n}\left(x\right)\right)\in l^{2}$,
$x\in X$ (\prettyref{lem:b3}), and 
\[
\left\langle \varphi\left(x\right),\varphi\left(y\right)\right\rangle _{\mathscr{L}}=\sum_{n}\overline{f_{n}\left(x\right)}f_{n}\left(y\right)=K\left(x,y\right),\quad x,y\in X.
\]
\end{proof}
\begin{example}
\label{exa:c4}Let $K:X\times X\rightarrow\mathbb{C}$ be p.d., and
$\mathscr{H}_{K}$ the associated RKHS. Let $\left\{ f_{n}\right\} $
be an ONB for $\mathscr{H}_{K}$. 
\begin{enumerate}
\item Let $\left\{ Z_{n}\right\} $ be a sequence of i.i.d. Gaussian random
variables, where $Z_{n}\sim N\left(0,1\right)$, realized on a probability
space $L^{2}\left(\Omega,\mathbb{P}\right)$. Here, one may take $\Omega=\prod_{\mathbb{N}}\mathbb{R}$,
equipped with the $\sigma$-algebra $\mathscr{C}$ generated by the
cylinder sets. Define 
\[
\varphi\left(x\right)=\sum_{n}f_{n}\left(x\right)Z_{n}\left(\cdot\right)\in L^{2}\left(\Omega,\mathbb{P}\right).
\]
Then $\left(\varphi,L^{2}\left(\Omega,\mathbb{P}\right)\right)\in H_{S}\left(K\right)$,
and 
\[
K\left(x,y\right)=\left\langle \varphi\left(x\right),\varphi\left(y\right)\right\rangle _{L^{2}\left(\Omega,\mathbb{P}\right)}.
\]
\item The above holds, in particular, when $K$ is the reproducing kernel
of the Bergman space $B_{2}\left(\Omega\right)$, $\Omega\subset\mathbb{C}^{n}$.
For $n=1$, $\Omega=\mathbb{D}$, 
\[
K\left(z,w\right)=\left(1-\overline{w}z\right)^{-2}=\sum_{n=0}^{\infty}\left(n+1\right)\overline{w}^{n}z^{n},
\]
where $\left\{ \sqrt{1+n}z^{n}\right\} _{n\in\mathbb{N}_{0}}$ is
an ONB for $\mathscr{H}_{K}$. Setting 
\[
\varphi\left(z\right)=\sum_{n=0}^{\infty}\sqrt{n+1}z^{n}Z_{n}\left(\cdot\right)\in L^{2}\left(\Omega,\mathbb{P}\right),
\]
then 
\[
K\left(z,w\right)=\left\langle \varphi\left(z\right),\varphi\left(w\right)\right\rangle _{L^{2}\left(\Omega,\mathbb{P}\right)}.
\]
\item Choose $\mathscr{L}$ to be any $L^{2}$-space, e.g., $\mathscr{L}=L^{2}\left(M,\mu\right)$.
Then, with 
\[
\varphi\left(x\right)=\sum_{n}f_{n}\left(x\right)\zeta_{n}\left(\cdot\right)\in L^{2}\left(M,\mu\right),
\]
we have
\[
K\left(x,y\right)=\int_{M}\overline{\varphi\left(x\right)}\varphi\left(y\right)d\mu=\left\langle \varphi\left(x\right),\varphi\left(y\right)\right\rangle _{L^{2}\left(\mu\right)}.
\]
Here the variable $m\in M$ is supposed. 
\end{enumerate}
\end{example}

\subsection{\label{subsec:3-1}A duality for feature selections}

In \prettyref{prop:c3}, the case when $\mathscr{L}$ is another RKHS
is of particular interest in the analysis below, as it offers a certain
symmetry between the two RKHSs and their feature selections. This
is stated as follows: 
\begin{prop}[duality]
 \label{prop:3-5}Let $K,L$ be p.d. kernels on $X\times X$, and
let $\mathscr{H}_{K},\mathscr{H}_{L}$ be the corresponding RKHSs.
Choose an ONB $\left\{ f_{n}\right\} $ for $\mathscr{H}_{K}$, and
$\left\{ g_{n}\right\} $ for $\mathscr{H}_{L}$. Define the following
vector-valued functions on $X$:
\begin{align*}
\varphi\left(x\right) & =\sum_{n}f_{n}\left(x\right)g_{n}\left(\cdot\right)\in\mathscr{H}_{L},\\
\psi\left(x\right) & =\sum_{n}f_{n}\left(\cdot\right)g_{n}\left(x\right)\in\mathscr{H}_{K}.
\end{align*}
Then, 
\begin{align*}
\left(\varphi,\mathscr{H}_{L}\right) & \in H_{S}\left(K\right),\;\text{and}\\
\left(\psi,\mathscr{H}_{K}\right) & \in H_{S}\left(L\right).
\end{align*}
\end{prop}

\begin{proof}
Note that $\varphi,\psi$ are well defined since, by \prettyref{lem:b3},
$\left(f_{n}\left(x\right)\right),\left(g_{n}\left(x\right)\right)\in l^{2}$
for all $x\in X$. Moreover, for all $x,y\in X$, 
\begin{align*}
\left\langle \varphi\left(x\right),\varphi\left(y\right)\right\rangle _{\mathscr{H}_{L}} & =\sum\overline{f_{n}\left(x\right)}f_{n}\left(y\right)=K\left(x,y\right),\;\text{and}\\
\left\langle \psi\left(x\right),\psi\left(y\right)\right\rangle _{\mathscr{H}_{K}} & =\sum\overline{g_{n}\left(x\right)}g_{n}\left(y\right)=L\left(x,y\right),
\end{align*}
which is the desired conclusion. 
\end{proof}
\begin{example}
\label{exa:c6}Consider the Szeg\H{o} kernel $K_{Sz}\left(x,y\right)=\left(1-\overline{w}z\right)^{-1}$,
$\left(w,z\right)\in\mathbb{D}\times\mathbb{D}$. Its RKHS is the
Hardy space $H_{2}\left(\mathbb{D}\right)=\left\{ \sum_{n=0}^{\infty}c_{n}z_{n}^{n}:\left(c_{n}\right)\in l^{2}\right\} $
with ONB $\left\{ z_{n}\right\} _{n\in\mathbb{N}_{0}}$. 

Let $K\left(x,y\right)=\left(1-xy\right)^{-1}=\sum_{n=0}^{\infty}x^{n}y^{n}$,
defined on $J\times J$, where $J=\left(-1,1\right)$. That is, $K=K_{Sz}\big|_{J\times J}$.
It follows from \prettyref{lem:b3} that $\left\{ x^{n}\right\} _{n\in\mathbb{N}_{0}}$
is an ONB for $\mathscr{H}_{K}$. Setting 
\[
\varphi\left(x\right)=\sum_{n\in\mathbb{N}_{0}}x^{n}z^{n},
\]
then $\left(\varphi,H_{2}\left(\mathbb{D}\right)\right)\in H_{S}\left(K\right)$
by \prettyref{prop:3-5}. 
\end{example}

\subsection{Operations}

A key property in the choice of Hilbert spaces when constructing feature
spaces for positive definite (p.d.) kernels is that tensor products
behave well; specifically, the category of Hilbert space is closed
under tensor product, meaning that the tensor product formed from
two Hilbert spaces is a new and canonically defined Hilbert space,
see e.g., \cite{MR3526117}. Here we take advantage of this geometric
fact, showing e.g., that if a p.d. kernel $M$ arises as a product
$M=KL$ (\prettyref{prop:b4}), then the feature spaces for $M$ arise
as tensor products of the feature spaces for the respective factors
$K$ and $L$.
\begin{prop}
\label{prop:3-7}Let $K_{i}$, $i=1,2$, be p.d. kernels on $X\times X$,
and let 
\begin{equation}
K\left(x,y\right)\coloneqq K_{1}\left(x,y\right)K_{2}\left(x,y\right),\quad x,y\in X,\label{eq:c2}
\end{equation}
the Hadamard product. Suppose 
\begin{equation}
\left(\varphi_{i},\mathscr{H}_{i}\right)\in H_{S}\left(K_{i}\right),\quad i=1,2.\label{eq:c3}
\end{equation}
Set $\mathscr{H}\coloneqq\mathscr{H}_{1}\otimes\mathscr{H}_{2}$ (as
tensor product in the category of Hilbert spaces), and 
\begin{equation}
\varphi\left(x\right)=\varphi_{1}\left(x\right)\otimes\varphi_{2}\left(x\right),\quad x\in X.\label{eq:c4}
\end{equation}
Then 
\[
\left(\varphi,\mathscr{H}\right)\in H_{S}\left(K\right).
\]
\end{prop}

\begin{proof}
We have
\begin{eqnarray*}
K\left(x,y\right) & \underset{\left(\ref{eq:c2}\right)}{=} & K_{1}\left(x,y\right)K_{2}\left(x,y\right)\\
 & \underset{\left(\ref{eq:c3}\right)}{=} & \left\langle \varphi_{1}\left(x\right),\varphi_{1}\left(y\right)\right\rangle _{\mathscr{H}_{1}}\left\langle \varphi_{2}\left(x\right),\varphi_{2}\left(y\right)\right\rangle _{\mathscr{H}_{2}}\\
 & \underset{\left(\ref{eq:c4}\right)}{=} & \left\langle \varphi\left(x\right),\varphi\left(y\right)\right\rangle _{\mathscr{H}_{1}\otimes\mathscr{H}_{2}}.
\end{eqnarray*}
\end{proof}
\begin{prop}
Let $I$ be an index set. Suppose $\left(\varphi_{i},\mathscr{H}_{i}\right)\in H_{S}\left(K_{i}\right)$,
$i\in I$, and $\sum_{i\in I}K_{i}\left(x,x\right)<\infty$ for all
$x\in X$. Let $K\coloneqq\sum_{i\in I}K_{i}$, $\mathscr{H}\coloneqq\oplus_{i\in I}\mathscr{H}_{i}$,
and $\varphi\left(x\right)\coloneqq\oplus_{i\in I}\varphi_{i}\left(x\right)$.
Then, $\left(\varphi,\mathscr{H}\right)\in H_{S}\left(K\right)$.
\end{prop}

\begin{proof}
Note $K$ is well defined if and only if $\sum_{i\in I}K_{i}\left(x,x\right)<\infty$,
for all $x\in X$. By assumptions, 
\[
\left\langle \varphi\left(x\right),\varphi\left(y\right)\right\rangle _{\mathscr{H}}=\sum_{i\in I}\left\langle \varphi_{i}\left(x\right),\varphi_{i}\left(y\right)\right\rangle _{\mathscr{H}_{i}}=\sum_{i\in I}K_{i}\left(x,y\right)=K\left(x,y\right).
\]
\end{proof}
\begin{lem}[sums of p.d. kernels]
 Let $K_{i}$, $i=1,2$, be p.d. kernels on $X\times X$. Then it
is immediate that $K=K_{1}+K_{2}$ is also p.d. Indeed with \prettyref{def:C1},
we have that for every $\left(\varphi_{i},\mathscr{L}_{i}\right)\in H_{S}\left(K_{i}\right)$,
then 
\[
\left(\varphi^{\oplus}\left(x\right)=\varphi_{1}\left(x\right)+\varphi_{2}\left(x\right),\mathscr{L}_{1}\oplus\mathscr{L}_{2}\right)\in H_{S}\left(K\right).
\]
\end{lem}

\begin{proof}
~
\begin{align*}
\left\langle \varphi^{\oplus}\left(x\right),\varphi^{\oplus}\left(y\right)\right\rangle _{\mathscr{L}_{1}\oplus\mathscr{L}_{2}} & =\left\langle \varphi_{1}\left(x\right),\varphi_{1}\left(y\right)\right\rangle _{\mathscr{L}_{1}}+\left\langle \varphi_{2}\left(x\right),\varphi_{2}\left(y\right)\right\rangle _{\mathscr{L}_{2}}\\
 & =K_{1}\left(x,y\right)+K_{2}\left(x,y\right)=K\left(x,y\right).
\end{align*}
 
\end{proof}
However, the RKHS from $K=K_{1}+K_{2}$ is not a direct sum Hilbert
space. By \cite{MR51437}, $F\in\mathscr{H}_{K}$ has its norm-represented
as
\[
\left\Vert F\right\Vert _{\mathscr{H}_{K}}^{2}=\inf_{F_{1},F_{2}}\left\{ \left\Vert F_{1}\right\Vert _{\mathscr{H}_{K_{1}}}^{2}+\left\Vert F_{1}\right\Vert _{\mathscr{H}_{K_{2}}}^{2},F_{1}+F_{2}=F,F_{i}\in\mathscr{H}_{K_{i}}\right\} .
\]
Indeed, the RKHS $\mathscr{H}_{K}$ takes the form 
\[
\mathscr{H}_{K}\simeq\left(\mathscr{H}_{K_{1}}\oplus\mathscr{H}_{K_{2}}\right)\ominus N
\]
where 
\[
N=\left\{ \left(f,-f\right)\in\mathscr{H}_{K_{1}}\oplus\mathscr{H}_{K_{2}},f\in\mathscr{H}_{K_{1}}\cap\mathscr{H}_{K_{2}}\right\} .
\]

Of special significance to the above discussion are the following
citations \cite{MR4689655, MR4751246, MR3454256, MR2597189, MR2810909, MR3024752, MR3875593}.

\section{\label{sec:4}Hilbert space of distributions}

Above, in Sections \ref{sec:2} and \ref{sec:3}, we introduced realizations
via feature maps. Here we address the of \textquotedblleft good\textquotedblright{}
choices of feature spaces, in particular, we make precise the choices
of \textquotedblleft bigger\textquotedblright{} features spaces, taking
the form of Hilbert spaces of Schwartz distributions. 

The details below concern special families of p.d. kernels $K$, and
ways to make precise the corresponding feature spaces, including the
form taken by the RKHS $\mathscr{H}_{K}$. This is motivated in part
by an important paper by Laurent Schwartz \cite{MR179587}.
\begin{thm}
\label{thm:d1}Let $K$ be a p.d. kernel on $\Omega\times\Omega$,
where $\Omega\subset\mathbb{R}^{d}$ is open, and let $\mathscr{H}_{K}$
be the corresponding RKHS. Suppose $K\in C^{\infty}\left(\Omega\times\Omega\right)$,
and let $\left\{ f_{n}\right\} $ be an ONB for $\mathscr{H}_{K}$,
so that $K\left(x,y\right)=\sum_{n}\overline{f_{n}\left(x\right)}f_{n}\left(y\right)$,
where $f_{n}\in C^{\infty}$.

Let $\mathscr{E}'\left(\Omega\right)$ denote the space of Schwartz
distributions with compact support in $\Omega$, i.e., $\mathscr{E}'\left(\Omega\right)$
is the Frechet dual of $C^{\infty}\left(\Omega\right)$. Let 
\begin{equation}
H_{K}^{dist}\left(\Omega\right)=\left\{ D\in\mathscr{E}'\left(\Omega\right):DKD<\infty\right\} \label{eq:d41}
\end{equation}
with an ONB $\left\{ D_{n}\right\} $. The notation $DKD$ in \eqref{eq:d41}
refers to a pair $K,D$ where $K$ is a $C^{\infty}$ p.d. kernel
and $D$ is a Schwartz distribution. Then $DKD$ refers to the action
of $D$ in the two variables of $K$, i.e., on the left and on the
right; see the cited literature. Here, the Hilbert completion is with
respect to the inner product $\left(\xi,\eta\right)\mapsto\xi K\eta$. 

Set 
\[
\varphi\left(x\right)=\sum_{n}f_{n}\left(x\right)D_{n}\in H_{K}^{dist}\left(\Omega\right),
\]
then 
\[
\left(\varphi,H_{K}^{dist}\left(\Omega\right)\right)\in H_{S}\left(K\right),
\]
i.e., 
\[
K\left(x,y\right)=\left\langle \varphi\left(x\right),\varphi\left(y\right)\right\rangle _{H_{K}^{dist}\left(\Omega\right)}.
\]
\end{thm}

\begin{proof}
See \prettyref{prop:c3}. 
\end{proof}
\begin{example}[\cite{MR1895530}]
\label{exa:d2}Let $K\left(x,y\right)=\left(1-xy\right)^{-1}$, defined
on $J\times J$, where $J=\left(-1,1\right)$. The corresponding RKHS
$\mathscr{H}_{K}$ has an ONB $\left\{ x^{n}:n\in\mathbb{N}_{0}\right\} $.

Let $D_{n}=\left(n!\right)^{-1}\delta_{0}^{\left(n\right)}$, where
$\delta_{0}^{\left(n\right)}$ is the $n^{th}$ derivative of the
Dirac distribution $\delta_{0}$. Note that $K\in C^{\infty}$, and
\[
D_{n}\left(x\right)K\left(x,y\right)=y^{n}\in\mathscr{H}_{K}.
\]
Define
\begin{equation}
D_{n}KD_{m}\coloneqq D_{n}\left(x\right)K\left(x,y\right)D_{m}\left(y\right)=\begin{cases}
1 & n=m,\\
0 & n\neq m.
\end{cases}
\end{equation}
Let $\mathscr{H}_{K}^{dist}$ be the Hilbert completion of $span\left\{ D_{n}\right\} $,
then 
\begin{equation}
\mathscr{H}_{K}^{dist}=\left\{ \xi\in\mathscr{E}'\left(J\right):\xi K\xi<\infty\right\} =\left\{ \sum c_{n}D_{n}:\sum\left|c_{n}\right|^{2}<\infty\right\} ,\label{eq:d2}
\end{equation}
for which $\left\{ D_{n}:n\in\mathbb{N}_{0}\right\} $ is an ONB.
 Set 
\[
\varphi\left(x\right)=\sum_{n=0}^{\infty}x^{n}D_{n}\in\mathscr{H}_{K}^{dist},
\]
then $\left(\varphi,\mathscr{H}_{K}^{dist}\right)\in H_{S}\left(K\right)$,
by \prettyref{thm:d1}. Note that
\[
\mathscr{H}_{K}^{dist}\cap\mathscr{H}_{K}=\emptyset.
\]

Similarly, for any kernel that is defined by power series, we obtain
$\mathscr{H}_{K}^{dist}$ as in \eqref{eq:d2}, by adjusting the coefficients
of $\delta^{\left(n\right)}$. In particular, this applies to the
kernel
\[
K\left(x,y\right)=\left(1-xy\right)^{-n}=1+nxy+\frac{1}{2}n\left(n+1\right)x^{2}y^{2}+\cdots,\quad\left(x,y\right)\in J\times J.
\]
\end{example}

\begin{cor}
\label{cor:d3}Suppose $f\left(x\right)=\sum_{n=0}^{\infty}a_{n}x^{n}$,
$a_{n}>0$, with radius of convergence $r^{2}>0$. Let $K\left(x,y\right)=\sum_{n=0}^{\infty}a_{n}x^{n}y^{n}$,
defined on $J\times J$, where $J=\left(-r,r\right)$. Set 
\[
H_{K}^{dist}=\left\{ \sum_{n=0}^{\infty}c_{n}D_{n}:\sum\left|c_{n}\right|^{2}<\infty,\;D_{n}=\frac{1}{n!\sqrt{a_{n}}}\delta^{\left(n\right)}\right\} 
\]
and 
\[
\varphi\left(x\right)=\sum_{n=0}^{\infty}\sqrt{a_{n}}x^{n}D_{n}.
\]
Then $\left(\varphi,H_{K}^{dist}\right)\in H_{S}\left(K\right)$. 
\end{cor}

Of special significance to the above discussion are the following
citations \cite{MR4302453, MR4405120, MR1895530, MR179587}.

\section{\label{sec:E}Ordering of kernels and RKHSs}

Since the choice of \textquotedblleft good\textquotedblright{} features
for kernel-machines depend on prior identification of kernels, it
is clear that precise comparisons of kernels will be important. We
stress ordering of kernels. Their role is addressed below, addressing
the role played for feature selection by ordering between pairs kernels,
and their implications for computations. Details will be addressed
below.

The role of \textquotedblleft ordering\textquotedblright , as we saw,
arises for both issues dealing with feature selection, and from the
role kernels play in geometry and in analysis. More generally, the
question of order plays an important role in diverse methods used
for building new reproducing kernel Hilbert spaces (RKHSs) from other
Hilbert spaces with specified frame elements having specific properties.
Such new constructions of RKHSs are used in turn within the framework
of regularization theory, and in approximation theory; involving there
such questions as semiparametric estimation, and multiscale schemes
of regularization. Making use of the results from the previous two
section, we turn below to a systematic analysis of these questions
of ordering.

Returning to the general framework $(X,K)$, where the set $X$ does
not come with any particular structure, we now study the case when
$X$ is fixed, and we examine the collection of all p.d. kernels defined
on $X\times X$. A special feature of interest is that of deciding
how the ordering of pairs of p.d. kernels relates to operators which
map between the associated families of feature spaces, see \prettyref{thm:f6}.
This study includes an identification of multipliers, see \prettyref{cor:f5}.

We first recall Aronszajn\textquoteright s inclusion theorem, which
states that, for two p.d. kernels $K,L$ on $X\times X$, $K\leq L$
if and only if $\mathscr{H}_{K}$ is contractively contained in $\mathscr{H}_{L}$
(see e.g., \cite{MR51437}):
\begin{thm}
Let $K_{i}$, $i=1,2$, be p.d. kernels on $X\times X$. Then $\mathscr{H}_{K_{1}}\subset\mathscr{H}_{K_{2}}$
(bounded contained) if and only if there exists a constant $c>0$
such that $K_{1}\leq c^{2}K_{2}$. Moreover, $\left\Vert f\right\Vert _{\mathscr{H}_{2}}\leq c\left\Vert f\right\Vert _{\mathscr{H}_{1}}$
for all $f\in\mathscr{H}_{K_{1}}$. 
\end{thm}

This theorem is reformulated in \prettyref{lem:f1} by means of quadratic
forms. It is then extended in \prettyref{thm:f6} to feature spaces. 
\begin{lem}
\label{lem:f1}Suppose $K,L$ are p.d. on $X\times X$, and $K\leq L$.
Let $\mathscr{H}_{L}$ be the RKHS of $L$. 
\begin{enumerate}
\item \label{enu:t1}Let $L_{0}=span\left\{ L_{x}:x\in X\right\} $. Define
$\Phi:L_{0}\times L_{0}\rightarrow\mathbb{C}$ by 
\begin{equation}
\Phi\left(L_{x},L_{y}\right)=K\left(x,y\right)\label{eq:t1}
\end{equation}
and extend by linearity: 
\begin{equation}
\Phi\left(\sum_{i=1}^{m}c_{i}L_{x_{i}},\sum_{j=1}^{n}d_{j}L_{y_{j}}\right)=\sum_{i=1}^{m}\sum_{j=1}^{n}\overline{c_{i}}d_{j}K\left(x_{i},y_{j}\right).
\end{equation}
Then $\Phi$ extends to a bounded sesquilinear form on $\mathscr{H}_{L}$. 
\item \label{enu:t2}There exists a unique positive selfadjoint operator
$A$ on $\mathscr{H}_{L}$, such that $0\leq A\leq I$, and 
\begin{equation}
\Phi\left(f,g\right)=\left\langle A^{1/2}f,A^{1/2}g\right\rangle _{\mathscr{H}_{L}},\quad f,g\in\mathscr{H}_{L}.\label{eq:t3}
\end{equation}
\item Especially, 
\begin{equation}
K\left(x,y\right)=\left\langle A^{1/2}L_{x},A^{1/2}L_{y}\right\rangle _{\mathscr{H}_{L}},\quad x,y\in X.\label{eq:t4}
\end{equation}
\end{enumerate}
\end{lem}

\begin{proof}
Part \eqref{enu:t1}. Need only to show that 
\begin{align}
\left|\Phi\left(\sum_{i=1}^{m}c_{i}L_{x_{i}},\sum_{j=1}^{n}d_{j}L_{y_{i}}\right)\right|^{2} & \leq\left\Vert \sum_{i=1}^{m}c_{i}L_{x_{i}}\right\Vert _{\mathscr{H}_{L}}^{2}\left\Vert \sum_{j=1}^{n}d_{j}L_{y_{i}}\right\Vert _{\mathscr{H}_{L}}^{2}\label{eq:t5}
\end{align}
Assume $\left(\pi,\mathscr{E}\right)\in H_{S}\left(K\right)$, i.e.,
$K\left(x,y\right)=\left\langle \pi\left(x\right),\pi\left(y\right)\right\rangle _{\mathscr{E}}$.
Then,
\[
\left|\sum_{i=1}^{m}\sum_{j=1}^{n}\overline{c_{i}}d_{j}K\left(x_{i},y_{j}\right)\right|^{2}=\left|\left\langle \sum_{i=1}^{m}c_{i}\pi\left(x_{i}\right),\sum_{j=1}^{n}d_{j}\pi\left(y_{j}\right)\right\rangle _{\mathscr{E}}\right|^{2}.
\]
Thus,
\begin{align*}
\text{LHS}_{\left(\ref{eq:t5}\right)} & \leq\left\Vert \sum_{i=1}^{m}c_{i}\pi\left(x_{i}\right)\right\Vert _{\mathscr{E}}^{2}\left\Vert \sum_{j=1}^{n}d_{j}\pi\left(y_{j}\right)\right\Vert _{\mathscr{E}}^{2}\\
 & =\left(\sum_{s,t=1}^{m}\overline{c_{s}}c_{t}\left\langle \pi\left(x_{s}\right),\pi\left(x_{t}\right)\right\rangle _{\mathscr{E}}\right)\left(\sum_{s,t=1}^{n}\overline{d_{s}}d_{t}\left\langle \pi\left(y_{s}\right),\pi\left(y_{t}\right)\right\rangle _{\mathscr{E}}\right)\\
 & =\left(\sum_{s,t=1}^{m}\overline{c_{s}}c_{t}K\left(x_{s},x_{t}\right)\right)\left(\sum_{s,t=1}^{n}\overline{d_{s}}d_{t}K\left(y_{s},y_{t}\right)\right)\\
 & \leq\left(\sum_{s,t=1}^{m}\overline{c_{s}}c_{t}L\left(x_{s},x_{t}\right)\right)\left(\sum_{s,t=1}^{n}\overline{d_{s}}d_{t}L\left(y_{s},y_{t}\right)\right)=\text{RKS}_{\left(\ref{eq:t5}\right)}.
\end{align*}

Part \eqref{enu:t2} follows from the general theory of quadratic
forms. Note that \eqref{eq:t4} follows from \eqref{eq:t3} and \eqref{eq:t1}.
\end{proof}
\begin{cor}
\label{cor:f3}Assume $K,L$ are p.d. on $X\times X$, and $K\leq L$.
Let $\mathscr{H}_{K}$ and $\mathscr{H}_{L}$ be the corresponding
RKHSs. Then 
\[
K_{x}\mapsto A^{1/2}L_{x}
\]
extends to an isometry from $\mathscr{H}_{K}$ into $\mathscr{H}_{L}$. 
\end{cor}

\begin{rem}
Let a pair of p.d. kernels satisfy the Loewner order relation (\prettyref{def:order}.)
Note that then the corresponding operator $A$ introduced in \eqref{eq:t4}
and \prettyref{cor:f3} will be bounded. However, \prettyref{exa:f5}
below (see \eqref{eq:F8}) illustrates that, in general, the inverse
$A^{-1}$ will be an unbounded operator. In applications to the theory
of elliptic PDEs, the operator $A$ introduced in \eqref{eq:t4} and
\prettyref{cor:f3} may take the form of a \textquotedblleft Greens
function;\textquotedblright{} see e.g., \cite{MR95341,MR91442}.
\end{rem}

\begin{example}
\label{exa:f5}Let $K\left(z,w\right)=\left(1-\overline{w}z\right)^{-1}$,
$L\left(z,w\right)=\left(1-\overline{w}z\right)^{-2}$, defined on
$\mathbb{D}\times\mathbb{D}$, where 
\begin{align*}
\mathscr{H}_{K} & =H_{2}\left(\mathbb{D}\right)=\left\{ \sum_{n=0}^{\infty}c_{n}z^{n}:\left(c_{n}\right)\in l^{2}\right\} ,\\
\mathscr{H}_{L} & =B_{2}\left(\mathbb{D}\right)=\left\{ \sum_{n=0}^{\infty}c_{n}z^{n}:\left(c_{n}/\sqrt{1+n}\right)\in l^{2}\right\} .
\end{align*}
Define 
\begin{equation}
A\left(z^{n}\right)=\left(1+n\right)^{-1}z^{n}.\label{eq:F6}
\end{equation}
Then 
\begin{equation}
K\left(z,w\right)=\left\langle A^{1/2}L_{z},A^{1/2}L_{w}\right\rangle _{\mathscr{H}_{L}}.\label{eq:F7}
\end{equation}
Moreover, the inverse operator is given by 
\begin{equation}
A^{-1}=1+z\frac{d}{dz}:z^{n}\longmapsto\left(1+n\right)z^{n},\label{eq:F8}
\end{equation}
where $A^{-1}\geq1$. 
\end{example}

\begin{proof}[Proof of \eqref{eq:F7}]
Recall that $L_{w}\left(s\right)=L\left(s,w\right)=\sum_{n\in\mathbb{N}_{0}}\left(1+n\right)\overline{w}^{n}s^{n}$,
and 
\[
1=\left\Vert z^{n}\right\Vert _{\mathscr{H}_{K}}\geq\left\Vert z^{n}\right\Vert _{\mathscr{H}_{L}}=\frac{1}{\sqrt{1+n}},\quad n\in\mathbb{N}_{0}.
\]
Then, 
\begin{align*}
\left\langle A^{1/2}L_{z},A^{1/2}L_{w}\right\rangle _{\mathscr{H}_{L}} & =\left\langle A^{1/2}\sum_{n}\left(1+n\right)\overline{z}^{n}s^{n},A^{1/2}\sum_{m}\left(1+n\right)\overline{w}^{m}s^{m}\right\rangle _{\mathscr{H}_{L}}\\
 & =\sum_{n}\left(1+n\right)^{2}z^{n}\overline{w}^{n}\left\langle A^{1/2}s^{n},A^{1/2}s^{n}\right\rangle _{\mathscr{H}_{L}}\\
 & =\sum_{n}\left(1+n\right)z^{n}\overline{w}^{n}\left\langle s^{n},s^{n}\right\rangle _{\mathscr{H}_{L}}\\
 & =\sum_{n}z^{n}\overline{w}^{n}=\left\langle K_{z},K_{w}\right\rangle _{\mathscr{H}_{K}}=K\left(z,w\right).
\end{align*}
\end{proof}
\begin{rem}
If $j:\mathscr{H}_{K}\rightarrow\mathscr{H}_{L}$ is the inclusion
map, then the adjoint $j^{*}:\mathscr{H}_{L}\rightarrow\mathscr{H}_{K}$
is given by $j^{*}\left(z^{n}\right)=\left(1+n\right)^{-1}z^{n}$.
Therefore, the operator $A$ in \eqref{eq:F6} is precisely the contraction
$A=jj^{*}:\mathscr{H}_{L}\rightarrow\mathscr{H}_{L}$. 
\end{rem}

More generally, we have:
\begin{thm}
\label{thm:f6}Let $K,L$ be p.d. kernels on $X\times X$, with $\left(\varphi,\mathscr{K}\right)\in H_{S}\left(K\right)$
and $\left(\psi,\mathscr{L}\right)\in H_{S}\left(L\right)$. Then
$K\leq L$ if and only if there exists a positive selfadjoint operator
on $B$ on $\mathscr{L}$, such that $0\leq B\leq I$, and 
\[
K\left(x,y\right)=\left\langle \varphi\left(x\right),\varphi\left(y\right)\right\rangle _{\mathscr{K}}=\left\langle B^{1/2}\psi\left(x\right),B^{1/2}\psi\left(y\right)\right\rangle _{\mathscr{L}},\quad x,y\in X.
\]
\end{thm}

\begin{proof}
See the proof of \prettyref{lem:f1}. 
\end{proof}
\begin{cor}[Multipliers]
\label{cor:f5}Let $K$ be a p.d. kernel on $\mathbb{D}$ and $\mathscr{H}_{K}$
be the corresponding RKHS. For $\varphi$ in the unit ball $\left(H^{\infty}\right)_{1}$
of $H^{\infty}$, the function 
\[
K^{*}\left(z,w\right)=\left(1-\overline{\varphi\left(w\right)}\varphi\left(z\right)\right)K\left(z,w\right)
\]
is a p.d. kernel on $\mathbb{D}$ if and only if $\varphi$ is a contractive
multiplier on $\mathscr{H}$, i.e., $\left\Vert \varphi h\right\Vert _{\mathscr{H}_{K}}\leq\left\Vert h\right\Vert _{\mathscr{H}_{K}}$
for all $h\in\mathscr{H}_{K}$.
\end{cor}

\begin{proof}
The assertion is equivalent to: 
\begin{gather}
\overline{\varphi\left(w\right)}\varphi\left(z\right)K\left(z,w\right)\leq K\left(z,w\right)\label{eq:f6}\\
\Updownarrow\nonumber \\
\left\Vert \varphi h\right\Vert _{\mathscr{H}_{K}}\leq\left\Vert h\right\Vert _{\mathscr{H}_{K}},\forall h\in\mathscr{H}_{K}.\label{eq:f7}
\end{gather}
Pick an ONB $\left\{ f_{n}\right\} $ for $\mathscr{H}_{K}$, then
\begin{equation}
\overline{\varphi\left(w\right)}\varphi\left(z\right)K\left(z,w\right)=\sum_{n}\overline{\varphi\left(w\right)f_{n}\left(w\right)}\varphi\left(z\right)f_{n}\left(z\right)\label{eq:f1}
\end{equation}
Therefore, an application of \prettyref{thm:f6} to \eqref{eq:f1}
shows that \eqref{eq:f6} holds if and only if the operator $\mathscr{H}_{K}\rightarrow\mathscr{H}_{K}$,
$f_{n}\mapsto\varphi f_{n}$ is contractive, thus the equivalence
to \eqref{eq:f7}. 
\end{proof}
Next, we focus on certain limit constructions of p.d. kernels. 
\begin{defn}
Given a set $X$, let $Pos\left(X\right)$ be the set of all p.d.
kernels defined on $X\times X$. 
\end{defn}

\begin{thm}
Let $K_{1}\leq K_{2}\leq\cdots\leq K_{n}\leq K_{n+1}\leq\cdots$,
with $K_{n}\in Pos\left(X\right)$ for all $n\in\mathbb{N}$. Assume
that for all $x\in X$,
\begin{equation}
\sup_{n}K_{n}\left(x,x\right)=S\left(x\right)<\infty,\label{eq:f11}
\end{equation}
then the Hilbert completion 
\begin{equation}
\mathscr{H}_{\tilde{K}}:=\left(\bigcup_{n}\mathscr{H}_{K_{n}}\right)^{\sim},
\end{equation}
is the RKHS of a limit p.d. kernel 
\begin{equation}
\tilde{K}\left(x,y\right):=\lim_{n\rightarrow\infty}K_{n}\left(x,y\right),
\end{equation}
defined on $X\times X$. 
\end{thm}

\begin{proof}
First note that, from \eqref{eq:f11}, we get the following boundedness:
\[
\left|K_{n}\left({\color{blue}x,y}\right)\right|^{2}\leq K_{n}\left(x,x\right)K_{n}\left(y,y\right)\leq S\left(x\right)S\left(y\right)<\infty,
\]
and so the sequence 
\[
\left\{ K_{n}\left(x,y\right)\right\} _{n\in\mathbb{N}}
\]
is bounded in $\mathbb{C}$ for $\forall\left(x,y\right)\in X\times X$. 

For all $N\in\mathbb{N}$, $c_{i}\in\mathbb{C}$, $x_{i}\in X$, $1\leq i\leq N$,
set 
\begin{equation}
F_{n}\left(\vec{c},\vec{x},N\right)=\sum_{i}\sum_{j}\overline{c_{i}}c_{j}K_{n}\left(x_{i},x_{j}\right).
\end{equation}
Since $K_{n}\leq K_{n+1}$, it follows that 
\[
F_{n}\left(\vec{c},\vec{x},N\right)\leq F_{n+1}\left(\vec{c},\vec{x},N\right)
\]
and that 
\[
\sup_{n\in\mathbb{N}}F_{n}\left(\vec{c},\vec{x},N\right)<\infty\;\left(\text{by \ensuremath{\left(\ref{eq:f11}\right)}}\right).
\]
Hence $\tilde{K}$ is well defined, and 
\[
\sum_{i}\sum_{j}\overline{c_{i}}c_{j}\tilde{K}\left(x_{i},x_{j}\right)=\sup_{n}F_{n}\left(c,\vec{x},N\right)<\infty.
\]

In this case, for every $f\in\mathscr{H}_{K_{1}}\subset\mathscr{H}_{K_{2}}\subset\cdots$,
\begin{equation}
\left\Vert f\right\Vert _{\mathscr{H}_{K_{1}}}\geq\left\Vert f\right\Vert _{\mathscr{H}_{K_{2}}}\geq\left\Vert f\right\Vert _{\mathscr{H}_{K_{3}}}\label{eq:E16}
\end{equation}
and we have
\[
\left\Vert f\right\Vert _{\mathscr{H}_{\tilde{K}}}=\lim_{n}\left\Vert f\right\Vert _{\mathscr{H}_{K_{n}}}.
\]
\end{proof}
\begin{rem}
\label{rem:e11}Condition \eqref{eq:f11} is necessary for this construction.
For example, consider $K_{n}\left(z,w\right)=\left(1-\overline{w}z\right)^{-n}$
on $\mathbb{D}\times\mathbb{D}$, where $n\in\mathbb{N}$. Then 
\[
1=\left\Vert z^{k}\right\Vert _{\mathscr{H}_{K_{1}}}>\underset{\frac{1}{\sqrt{1+k}}}{\underbrace{\left\Vert z^{k}\right\Vert _{\mathscr{H}_{K_{2}}}}}>\left\Vert z^{k}\right\Vert _{\mathscr{H}_{K_{3}}}\rightarrow0,\;n\rightarrow\infty.
\]
\end{rem}

As an application we mention the following Cantor construction and
a monotone kernel limit. While the example selects a particular scaling-iteration,
the idea will apply more generally to a variety of iterated function
system constructions (IFSs). For background on IFSs, see e.g., \cite{MR4590528,MR4291375}.
\begin{lem}
\label{lem:f11}Let $f\in C\left(\left[0,1\right]\right)$, and extend
it to $\mathbb{R}$ by setting $f\left(x\right)=0$ for $x\notin\left[0,1\right]$.
Define $T^{0}f=f$, and 
\begin{equation}
T^{n}f\left(x\right)=T^{n-1}f\left(3x\right)+T^{n-1}f\left(3x-2\right),\quad n\in\mathbb{N}.\label{eq:f15}
\end{equation}
Then the limit (pointwise)
\begin{equation}
F\left(x\right)=\lim_{n\rightarrow\infty}T^{n}f\left(x\right)\label{eq:f16}
\end{equation}
is supported in the middle-third Cantor set $C_{1/3}$. (See \prettyref{fig:f1}
for an illustration.)
\end{lem}

\begin{figure}
\includegraphics[width=0.9\columnwidth]{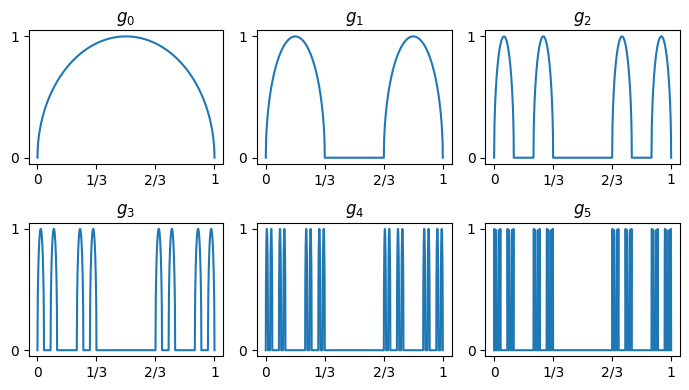}

\caption{\label{fig:f1}$g_{n}\left(x\right)=T^{n}f\left(x\right)$, $n=0,1,\cdots,5$.}
\end{figure}

\begin{proof}
Recall that $C_{1/3}$ is defined as follows: Let $I=\left[0,1\right]$.
Introduce two endomorphisms $\tau_{1},\tau_{2}:I\rightarrow I$, where
$\tau_{1}\left(x\right)=x/3$, $\tau_{2}\left(x\right)=\left(x+2\right)/3$.
Set $C_{0}=I$, and 
\[
C_{n}=\tau_{1}\left(C_{n-1}\right)\cup\tau_{2}\left(C_{n-1}\right),\quad n\in\mathbb{N}.
\]
Then 
\[
C_{1/3}=\bigcap_{n=0}^{\infty}C_{n}.
\]
Note \eqref{eq:f15} is the dual construction for functions on the
unit interval $I$. 
\end{proof}
\begin{thm}
Let $K$ be p.d. on $X\times X$ with $X=\left[0,1\right]$, such
that 
\[
K\left(x,y\right)=\sum_{i=0}^{\infty}f_{i}\left(x\right)\overline{f_{i}\left(y\right)},
\]
where $\left\{ f_{i}\right\} $ is an ONB for the corresponding RKHS
$\mathscr{H}_{K}$. Extend $f_{i}$ to $\mathbb{R}$ by setting $f_{i}\left(x\right)=0$
for $x\notin\left[0,1\right]$, and set 
\[
K_{n}\left(x,y\right)=\sum_{i=0}^{\infty}T^{n}f_{i}\left(x\right)\overline{T^{n}f_{i}\left(y\right)}.
\]
Then the limit 
\[
K_{\infty}\left(x,y\right)=\lim_{n\rightarrow\infty}K_{n}\left(x,y\right)=\lim_{n\rightarrow\infty}\sum_{i=0}^{\infty}T^{n}f_{i}\left(x\right)\overline{T^{n}f_{i}\left(y\right)}
\]
is a p.d. kernel on $C_{1/3}\times C_{1/3}$. 

Moreover, $K_{\infty}$ is invariant under the action of $T$, where
$T$ acts on a p.d. kernel $L$ on $X\times X$ by 
\[
L\left(x,y\right)=\sum_{i}l_{i}\left(x\right)\overline{l_{i}\left(y\right)}\longmapsto TL\left(x,y\right)=\sum_{i}Tl_{i}\left(x\right)\overline{Tl_{i}\left(y\right)},\quad x,y\in X.
\]
 
\end{thm}

\begin{proof}
By assumption, $\left(f_{i}\left(x\right)\right)\in l^{2}$, $\forall x\in\mathbb{R}$.
Note that 
\[
\left(f_{i}\left(x\right)\right)\in l^{2},\forall x\in\mathbb{R}\Longrightarrow\left(Tf_{i}\left(x\right)\right)\in l^{2},\forall x\in\mathbb{R}
\]
since 
\[
\left\Vert \left(Tf_{i}\left(x\right)\right)\right\Vert _{l^{2}}^{2}\leq\left\Vert \left(f_{i}\left(3x\right)\right)\right\Vert _{l^{2}}^{2}+\left\Vert \left(f_{i}\left(3x-2\right)\right)\right\Vert _{l^{2}}^{2}<\infty.
\]
Therefore, $K_{n}$ is a well defined p.d. kernel, for all $n\in\mathbb{N}$. 

The conclusion follows by passing to the limit, where $F_{i}\left(x\right):=\lim_{n\rightarrow\infty}T^{n}f_{i}\left(x\right)$
exists, and has support in $C_{1/3}$, by \eqref{eq:f15}--\eqref{eq:f16}.
See \prettyref{fig:f2} for an illustration.
\end{proof}
\begin{figure}
\includegraphics[width=0.8\columnwidth]{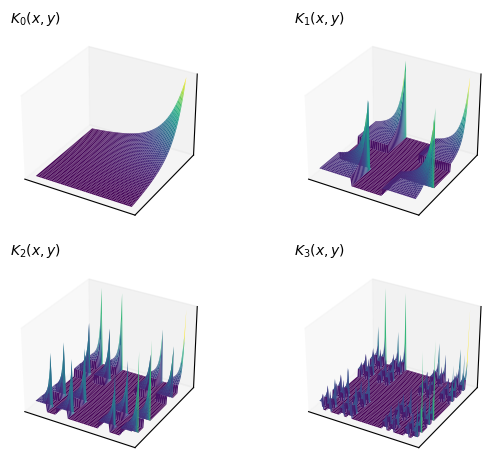}
\begin{align*}
K\left(x,y\right) & =\frac{1}{1-xy}=\sum_{i=0}^{\infty}x^{i}y^{i}=\sum_{i=0}^{\infty}f_{i}\left(x\right)f_{i}\left(y\right)\\
K_{n}\left(x,y\right) & =\sum_{i=0}^{\infty}T^{n}f_{i}\left(x\right)T^{n}f_{i}\left(y\right),\quad x,y\in\left(0,1\right).
\end{align*}

\caption{\label{fig:f2}$K_{n}$, $n=0,1,2,3$.}
\end{figure}

Of special significance to the above discussion are the following
citations \cite{MR4561157, MR4295177, MR4294330, MR3173967, MR0486556}.

\section{\label{sec:F}RKHS of analytic functions}

As noted in \prettyref{sec:4}, an identification of good kernels,
and their corresponding RKHSs, depend on the particular function spaces
that arise as RKHSs. The choices when the RKHSs consist of Hilbert
spaces of analytic functions has received special attention in the
earlier literature on the use of kernels in analysis. The section
below outlines properties of RKHSs realized as Hilbert spaces of analytic
functions, and their role in our present applications.

The focus of our analysis below is the case when the RKHS $\mathscr{H}_{K}$
will be Hilbert spaces of analytic function, defined on an open domain
in $\mathbb{C}^{d}$ for some $d$.
\begin{defn}
\label{def:d1}Let $\Omega$ be an open subset in $\mathbb{C}^{d}$
and let $K$ be a $\mathbb{C}$-valued p.d. function on $\Omega\times\Omega$.
We say that $K$ is analytic if the corresponding RKHS $\mathscr{H}_{K}$
consists of analytic functions on $\Omega$.
\end{defn}

\begin{rem}
We note that there are other definitions in the literature which make
precise this property of analyticity, and it follows from our discussion
that they are equivalent to the present one. 

Note that \prettyref{def:d1} makes it clear that the following three
familiar classes of p.d. kernels $K$ are analytic: The cases when
$K$ is a Szeg\H{o} kernel, or a Bergman kernel, or Bargmann\textquoteright s
kernel \cite{MR75656,MR4586941,MR3658810}. In these cases, the respective
RKHSs $\mathscr{H}_{K}$ are the Hardy space $H_{2}\left(\Omega\right)$,
the Bergman space $B_{2}\left(\Omega\right)$, or Bargmann\textquoteright s
Hilbert space of entire analytic functions on $\mathbb{C}^{d}$, also
called the Segal-Bargmann space. For the literature, we refer to \cite{MR4122067,MR3410523,MR2012838,MR4553935,MR4248024,MR3603383},
and we call attention to the Drury-Arveson kernel \cite{MR1668582}
as generalization of the Szego/Bergman case. 
\end{rem}

\begin{example}[$\text{Bergman = \ensuremath{\left(\text{Szeg\H{o}}\right)^{2}}}$]
\label{exa:f3} Recall the Szeg\H{o} kernel 
\[
K\left(z,w\right)=\sum_{n\in\mathbb{N}_{0}}z^{n}\overline{w}^{n}
\]
and the Bergman kernel 
\[
K^{2}\left(z,w\right)=\sum_{n\in\mathbb{N}_{0}}\left(n+1\right)z^{n}\overline{w}^{n}.
\]
Here, we have 
\[
K^{2}-K=K\left(K-1\right)=\left(\frac{1}{1-z\overline{w}}\right)^{2}\left(z\overline{w}\right)\geq0,
\]
i.e., $K\leq K^{2}$. By the discussion above, we have 
\[
B_{2}\left(\mathbb{D}\right)\xrightarrow{\quad A\quad}H_{2}\left(\mathbb{D}\right)\subset B_{2}\left(\mathbb{D}\right),
\]
where $A$ is the operator in \eqref{eq:A1}. Specifically, 
\end{example}

\begin{equation}
A\underset{\in B_{2}}{\underbrace{\left(\sum_{n=0}^{\infty}c_{n}z^{n}\right)}}:=\sum_{n=0}^{\infty}\frac{c_{n}}{\sqrt{n+1}}z^{n}\in H_{2}\label{eq:A1}
\end{equation}
where
\begin{align*}
\left\Vert \sum_{n=0}^{\infty}\frac{c_{n}}{\sqrt{n+1}}z^{n}\right\Vert _{H_{2}}^{2} & =\sum_{n=0}^{\infty}\frac{\left|c_{n}\right|^{2}}{n+1}\\
\left\Vert \sum_{n=0}^{\infty}c_{n}z^{n}\right\Vert _{B_{2}}^{2} & =\left\Vert \sum_{n=0}^{\infty}\frac{c_{n}}{\sqrt{n+1}}\sqrt{n+1}z^{n}\right\Vert _{B_{2}}^{2}=\sum_{n=0}^{\infty}\frac{\left|c_{n}\right|^{2}}{n+1}.
\end{align*}

\begin{rem}
\label{rem:f4}Note this covers a lot of the kernels we considered,
such as
\end{rem}

\begin{enumerate}
\item $K$: 
\[
\begin{matrix}{\displaystyle \frac{1}{1-z\overline{w}}}, & {\displaystyle \left(\frac{1}{1-z\overline{w}}\right)^{2}}, & X=\mathbb{D}\end{matrix}
\]
\item $K:$ 
\[
\frac{1}{2i\left(z-\overline{w}\right)}
\]
defined for $\left(z,w\right)\in\mathbb{C}_{+}\times\mathbb{C}_{+}$,
where 
\[
\mathbb{C}_{+}=\left\{ z\in\mathbb{C}:\Im z>0\right\} 
\]
\item \label{enu:f4-3}The Bargmann kernel: 
\[
e^{z\overline{w}}=\sum_{n=0}^{\infty}\frac{z^{n}\overline{w}^{n}}{n!}=\sum_{n=0}^{\infty}\frac{z^{n}}{\sqrt{n!}}\frac{\overline{w}^{n}}{\sqrt{n!}},\quad\forall\left(z,w\right)\in\mathbb{C}\times\mathbb{C}.
\]
This leads to an RKHS $\mathscr{H}_{K}$ consisting of all entire
functions $F$ on $\mathbb{C}$ with norm 
\begin{equation}
\int_{\mathbb{C}}\left|F\left(z\right)\right|^{2}e^{-\frac{\left|z\right|^{2}}{2}}dA\left(z\right)<\infty.
\end{equation}
Here, $dA\left(z\right)=dxdy$, $z=x+iy$. 
\end{enumerate}
\begin{rem}
Comparing kernels is relatively straightforward, 
\begin{equation}
K\leq K'\Longleftrightarrow\text{\ensuremath{K'-K} is p.d.}
\end{equation}
while comparing the associated Hilbert spaces is intriguing. The challenge
lies in understanding how the embedding or inclusion of one RKHS into
another reflects the geometry of the underlying kernels. For instance,
the inclusion involves not just the kernels' positivity but also their
interaction with the data, operator norms, and potential scaling factors.
Moreover, the relationship between the norms of the two spaces is
critical, as it determines the stability and sensitivity of algorithms
using these spaces. This makes the comparison of RKHSs more than a
direct numerical or functional comparison---it becomes a study of
their geometry, boundedness properties, and the behavior of operators
that map between them. 

For example, 
\begin{equation}
\text{\ensuremath{H_{2}\left(\mathbb{D}\right)} (Hardy space) vs \ensuremath{B_{2}\left(\mathbb{D}\right)} (Bergman space),}
\end{equation}
see the discussion above.
\end{rem}

Now, consider feature maps and feature spaces:
\begin{equation}
H_{S}\left(K\right)=\left\{ \left(\varphi,\mathscr{H}\right),X\ni x\rightarrow\varphi\left(x\right)\in\mathscr{H},\:\text{s.t. \ensuremath{K\left(x,y\right)=\left\langle \varphi\left(x\right),\varphi\left(y\right)\right\rangle _{\mathscr{H}}}}\right\} .
\end{equation}
We may also consider the following two variants of $H_{S}\left(K\right)$:
\begin{defn}
Given $K$, p.d. in $X\times X$, set
\begin{multline}
\text{super feature space:}\\
H_{S}^{+}\left(K\right)=\left\{ \left(\psi,\mathscr{H}\right);X\ni x\rightarrow\psi\left(x\right)\in\mathscr{H},\:\text{s.t. \ensuremath{K\leq\left\langle \psi\left(x\right),\psi\left(y\right)\right\rangle _{\mathscr{H}}}}\right\} ,
\end{multline}
\begin{multline}
\text{sub feature space: }\\
H_{S}^{-}\left(K\right)=\left\{ \left(\psi,\mathscr{H}\right);X\ni x\rightarrow\psi\left(x\right)\in\mathscr{H},\:\text{s.t. \ensuremath{\left\langle \psi\left(x\right),\psi\left(y\right)\right\rangle _{\mathscr{H}}\leq K}}\right\} .
\end{multline}
\end{defn}

\subsection{Three kernels of $H_{2}$-Hardy spaces}

Here, $f_{n}\left(z\right)=z^{n}$, in a complex variable $z\in\mathbb{D}=\left\{ z\in\mathbb{D}:\left|z\right|<1\right\} $,
the unit disk in $\mathbb{C}$. 

Summary: 
\begin{enumerate}
\item Coefficients in the scalar $\mathbb{C}$:
\[
H_{2}\left(\mathbb{D}\right)=\left\{ \sum_{n=0}^{\infty}c_{n}z^{n}:\left(c_{n}\right)\in l^{2}\left(\mathbb{N}_{0}\right)\right\} ,
\]
\[
\left\Vert \sum_{n=0}^{\infty}c_{n}z^{n}\right\Vert _{H_{2}\left(\mathbb{D}\right)}^{2}=\sum_{n=0}^{\infty}\left|c_{n}\right|^{2}=\left\Vert \left(c_{n}\right)\right\Vert _{l^{2}}^{2}.
\]
\item Coefficients in a fixed Hilbert space $\mathscr{H}$:
\[
H_{2}\left(\mathscr{H}\right):=\left\{ \sum_{n=0}^{\infty}h_{n}z^{n}:h_{n}\in\mathscr{H},\:\sum_{n=0}^{\infty}\left\Vert h_{n}\right\Vert ^{2}<\infty\right\} ,
\]
\[
\left\Vert \sum_{n=0}^{\infty}h_{n}z^{n}\right\Vert _{H_{2}\left(\mathscr{H}\right)}^{2}=\sum_{n=0}^{\infty}\left\Vert h_{n}\right\Vert _{\mathscr{H}}^{2}<\infty.
\]
\item \label{enu:f3}Coefficients in the Hilbert-Schmidt class $HS\left(\mathscr{H}\right)\subset B\left(\mathscr{H}\right)$,
where $B\left(\mathscr{H}\right)$ is the space of all bounded operators
in $\mathscr{H}$: 
\begin{multline*}
H_{2}\left(B\left(\mathscr{H}\right)\right):=\left\{ \sum_{n=0}^{\infty}Q_{n}z^{n}:Q_{n}\in B\left(\mathscr{H}\right),\right.\\
\;\left.\sum_{n=0}^{\infty}Q_{n}^{*}Q_{n}\in\mathscr{T}\left(\mathscr{H}\right),\:\text{trace class}\right\} ,
\end{multline*}
\[
\left\Vert \sum_{n=0}^{\infty}Q_{n}z^{n}\right\Vert _{H_{2}\left(B\left(\mathscr{H}\right)\right)}^{2}=\text{Trace}\left(\sum_{n=0}^{\infty}Q_{n}^{*}Q_{n}\right).
\]
\end{enumerate}
Correspondences, transforms: $3\rightarrow2\rightarrow1$, $3\rightarrow1$.
Recall a Kaczmarz system of projections $P_{n}$ yields operators
$Q_{n}$ s.t. $\sum_{n}Q_{n}^{*}Q_{n}=I$. See e.g., \cite{MR4561157,MR4472249,MR4126821}
for additional details.

\subsection{Realization using tensor product of Hilbert spaces}

Recall that 
\begin{equation}
{\color{blue}\mathscr{H}\otimes\mathscr{H}^{*}}\longleftrightarrow\text{the Hilbert space of all \emph{Hilbert-Schmidt} operators acting on \ensuremath{\mathscr{H}.}}
\end{equation}
Consider case \eqref{enu:f3} from above, i.e., $F\left(z\right)=\sum_{n=0}^{\infty}Q_{n}z^{n}$,
then for $h\in\mathscr{H}$, 
\[
\left\langle h,F\left(z\right)h\right\rangle _{\mathscr{H}}=\sum_{n=0}^{\infty}\left\langle h,Q_{n}h\right\rangle _{\mathscr{H}}z^{n}\in H_{2}\left(\mathbb{D}\right),
\]
where $\left\langle \cdot,\cdot\right\rangle _{\mathscr{H}}$ denotes
the inner product in $\mathscr{H}$, and 
\begin{align*}
\left\Vert F\left(z\right)h\right\Vert _{\mathscr{H}}^{2} & =\left\Vert \sum_{n=0}^{\infty}\left(Q_{n}h\right)z^{n}\right\Vert _{H_{2}\left(\mathbb{D},\mathscr{H}\right)}^{2}\\
 & =\sum_{n=0}^{\infty}\left\Vert Q_{h}h\right\Vert _{\mathscr{H}}^{2}\left|z\right|^{2n}\\
 & =\sum_{n=0}^{\infty}\left\langle h,Q_{n}^{*}Q_{n}h\right\rangle _{\mathscr{H}}\left|z\right|^{2n}\\
 & \leq\left\Vert h\right\Vert _{\mathscr{H}}^{2}\sum_{n=0}^{\infty}\left\Vert Q_{n}^{*}Q_{n}\right\Vert \left|z\right|^{2n}\\
 & =\left\Vert h\right\Vert _{\mathscr{H}}^{2}\sum_{n=0}^{\infty}\left\Vert Q_{n}\right\Vert ^{2}\left|z\right|^{2n}
\end{align*}
where $\left\Vert \cdot\right\Vert =\left\Vert \cdot\right\Vert _{\mathscr{H}\rightarrow\mathscr{H}}$
is the operator norm. 

Trace-norm: Pick an ONB $\left\{ e_{k}\right\} $ in $\mathscr{H}$,
then 
\begin{align*}
\text{Tr}\left(F\left(z\right)^{*}F\left(z\right)\right) & =\sum_{k}\left\langle e_{k},F\left(z\right)^{*}F\left(z\right){\color{blue}e_{k}}\right\rangle \\
 & =\sum_{k}\left\Vert F\left(z\right)e_{k}\right\Vert _{\mathscr{H}}^{2}\\
 & =\sum_{n=0}^{\infty}\text{Tr}\left(Q_{n}^{*}Q_{n}\right)\left|z\right|^{2n}<\infty
\end{align*}
for $\forall z\in\mathbb{D}$ when $\sum Q_{n}^{*}Q_{n}$ is trace-class. 

Of special significance to the above discussion are the following
citations \cite{MR2012838, MR3410523, MR1668582, MR4472249, MR24487, MR3526117}.

\section{\label{sec:G}$K$-duality via the RKHS $\mathscr{H}_{K}$}

The present final section addresses a list of direct links between
kernel properties, and the role they play in feature selection. 

\subsection{Dirac-masses and $\mathscr{H}_{K}$}

In the below we consider the role of the Dirac masses in reproducing
kernel Hilbert space $\mathscr{H}_{K}$ when $K$ is a general p.d.
kernel defined on $X\times X$. Specifically, we show that the completion
of the span of the $X$-Dirac masses identifies as a realization of
all bounded linear functionals on $\mathscr{H}_{K}$.
\begin{thm}
\label{thm:7-1}Fix $K$, assumed p.d. on $X$. Let 
\begin{equation}
\mathscr{H}_{K}=\overline{span}^{\left\Vert \cdot\right\Vert _{\mathscr{H}_{K}}}\left\{ K_{x}:x\in X\right\} ,\quad\tilde{\mathscr{H}}_{K}=\overline{span}^{\left\Vert \cdot\right\Vert _{\tilde{\mathscr{H}}_{K}}}\left\{ \delta_{x}:x\in X\right\} ,
\end{equation}
where 
\begin{equation}
\left\Vert \sum\nolimits_{i}c_{i}K_{x_{i}}\right\Vert _{\mathscr{H}_{K}}^{2}=\left\Vert \sum\nolimits_{i}c_{i}\delta_{x_{i}}\right\Vert _{\tilde{\mathscr{H}}_{K}}^{2}=\sum\nolimits_{i,j}\overline{c_{i}}c_{j}K\left(x_{i},x_{j}\right).\label{eq:g2-1}
\end{equation}
Then
\begin{equation}
\tilde{\mathscr{H}}_{K}\simeq\mathscr{H}_{K}'.\label{eq:g3}
\end{equation}
\end{thm}

\begin{proof}
Every bounded linear functional $l$ on $\mathscr{H}_{K}$ is given
by a unique $\xi_{l}$ in $\mathscr{H}_{K}$, and $\xi_{l}$ is the
limit of a sequence $\left(\varphi_{n}\right)$ in $span\left\{ K_{x}:x\in X\right\} $.
Using the correspondence 
\[
\varphi_{n}\left(x\right)=\sum c_{i}^{\left(n\right)}K_{x_{i}}\xleftrightarrow{\text{by \ensuremath{\left(\ref{eq:g2-1}\right)}}}\tilde{\varphi}_{n}\left(x\right)=\sum c_{i}^{\left(n\right)}\delta_{x_{i}},
\]
$\left(\tilde{\varphi}_{n}\right)$ is Cauchy in $\tilde{\mathscr{H}}_{K}$,
and it converges to some $f\in\tilde{\mathscr{H}}_{K}$. Note, by
\eqref{eq:g2-1}, 
\[
\left\Vert \varphi_{n}-\xi_{l}\right\Vert _{\mathscr{H}_{K}}=\left\Vert \tilde{\varphi}_{n}-f\right\Vert _{\tilde{\mathscr{H}}_{K}}.
\]

Then, for all $h\in\mathscr{H}_{K}$, 
\[
l\left(h\right)=\left\langle \xi_{l},h\right\rangle _{\mathscr{H}_{K}}=\lim_{n}\left\langle \varphi_{n},h\right\rangle _{\mathscr{H}_{K}}=\lim_{n}\tilde{\varphi}_{n}\left(h\right)=f\left(h\right),
\]
where 
\[
\tilde{\varphi}_{n}\left(h\right)=\left(\sum c_{i}^{\left(n\right)}\delta_{x_{i}}\right)\left(h\right)=\left\langle \left(\sum c_{i}^{\left(n\right)}K_{x_{i}}\right),h\right\rangle _{\mathscr{H}_{K}}=\sum c_{i}^{\left(n\right)}h\left(x_{i}\right).
\]
Therefore, 
\[
l=f.
\]
This shows that $\mathscr{H}_{K}'\subset\tilde{\mathscr{H}}_{K}$.
Similarly, $\tilde{\mathscr{H}}_{K}\subset\mathscr{H}_{K}'$, and
so \eqref{eq:g3} holds. 
\end{proof}
We note that \prettyref{thm:7-1} follows from the Riesz Representation
Theorem. However, we include this theorem to explicitly establish
the equivalence between $\mathscr{H}_{K}'$, the dual space of $\mathscr{H}_{K}$,
and $\tilde{\mathscr{H}}_{K}$, the completion of the span of Dirac
masses. The point is to explicitly show how the RKHS structure and
kernel properties are used for this identification (see \eqref{eq:g3}).
In \prettyref{cor:g3} below, we use this to construct explicit bases
for $\mathscr{H}_{K}$ and $\mathscr{H}_{K}'$ for the kernel $K^{n}\left(x,y\right):=\left(1-xy\right)^{-n}$,
where the representation of the Dirac delta function gives a concrete
realization of $\tilde{\mathscr{H}}_{K}$. 
\begin{cor}
Fix $K,L$ p.d. on $X$. The following are equivalent:
\begin{enumerate}
\item $K\leq L$.
\item $\mathscr{H}_{K}$ is contractively contained in $\mathscr{H}_{L}$.
\item $\tilde{\mathscr{H}}_{L}$ is contractively contained in $\tilde{\mathscr{H}}_{K}$.
\end{enumerate}
Moreover, $\mathscr{H}_{K}$ is dense in $\mathscr{H}_{L}$ if and
only if $\tilde{\mathscr{H}}_{L}$ is dense in $\tilde{\mathscr{H}}_{K}$. 
\end{cor}

\begin{cor}
\label{cor:g3}Let $X=\left(-1,1\right)$, and 
\[
K^{n}\left(x,y\right)=\left(1-xy\right)^{-n}=\sum_{k=0}^{\infty}a_{k}x^{k}y^{k},\quad x,y\in X
\]
where 
\[
a_{k}=\left(-1\right)^{k}\binom{-n}{k}=\frac{n\left(n+1\right)\cdots\left(n+k-1\right)}{k}>0.
\]
Then 
\begin{align}
\mathscr{H}_{K^{n}} & =\left\{ \sum c_{k}x^{k}:\left(c_{k}/\sqrt{a_{k}}\right)\in l^{2}\right\} \simeq\left\{ \left(c_{k}\right):\left(c_{k}/\sqrt{a_{k}}\right)\in l^{2}\right\} \\
\mathscr{H}'_{K^{n}} & =\left\{ \sum c_{k}\frac{D^{k}}{k!}:\left(c_{k}\sqrt{a_{k}}\right)\in l^{2}\right\} \simeq\left\{ \left(c_{k}\right):\left(c_{k}\sqrt{a_{k}}\right)\in l^{2}\right\} 
\end{align}
where 
\[
D^{k}=\left(\frac{d}{dy}\right)^{k}\big|_{y=0}.
\]
Further, for all $x\in X$, 
\begin{equation}
\delta_{x}=\sum_{k=0}^{\infty}x^{k}\frac{D^{k}}{k!}.\label{eq:g4}
\end{equation}
\end{cor}

\begin{proof}
Note that
\[
\frac{D^{k}}{k!}\left(x^{m}\right)=\begin{cases}
1 & k=m\\
0 & k\neq m
\end{cases}
\]
and so we have the natural isomorphism 
\[
\mathscr{H}'_{K^{n}}\ni\frac{D^{k}}{k!}\longleftrightarrow a_{k}x^{k}\in\mathscr{H}_{K^{n}}.
\]
Thus, 
\[
\left\{ \frac{D^{k}}{k!\sqrt{a_{k}}}\right\} _{k=0}^{\infty}
\]
is an ONB for $\mathscr{H}'_{K^{n}}$. (See also \prettyref{cor:d3}.) 

Lastly, with $\delta_{x}\longleftrightarrow K_{x}$, then 
\[
\left\langle \frac{D^{k}}{k!\sqrt{a_{k}}},\delta_{x}\right\rangle _{\mathscr{H}'_{K^{n}}}=\left\langle \sqrt{a_{k}}x^{k},K_{x}\right\rangle _{\mathscr{H}_{K^{n}}}=\sqrt{a_{k}}x^{k},
\]
and thus 
\[
\delta_{x}=\sum_{k=0}^{\infty}\left\langle \frac{D^{k}}{k!\sqrt{a_{k}}},\delta_{x}\right\rangle _{\mathscr{H}'_{K^{n}}}\frac{D^{k}}{k!\sqrt{a_{k}}}=\sum_{k=0}^{\infty}\sqrt{a_{k}}x^{k}\frac{D^{k}}{k!\sqrt{a_{k}}}=\sum_{k=0}^{\infty}x^{k}\frac{D^{k}}{k!},
\]
which is \eqref{eq:g4}.
\end{proof}

\subsection{A $K$-transform and $K^{-1}$}

The following result is motivated by the special case when p.d. kernels
arise as Greens functions. In more detail, recall from the context
of PDEs, Greens functions arise as \textquotedblleft inverse\textquotedblright{}
to positive elliptic operators, see e.g., \cite{MR91442,MR4294330,MR1698952,MR1642132}.
Since, for these cases, therefore p.d. kernels $K$ arise as inverses
of elliptic PDEs, it seems natural, in our present general framework
of $Pos(X)$, to ask for a precise form of $K^{-1}$ .

Starting with $K\in Pos\left(X\right)$, and introduce functions $f\in\mathscr{H}_{K}$
(the RKHS), and signed measures $\mu$ on $X$ s.t. $\mu K\mu<\infty$. 

Recall that the following are equivalent:
\begin{gather}
\infty>\iint\mu\left(dx\right)K\left(x,y\right)\mu\left(dy\right)=\left(\text{abbreviated \ensuremath{\mu K\mu}}\right)\label{eq:fu1}\\
\Updownarrow\nonumber \\
\int\mu\left(dx\right)K\left(x,\cdot\right)\in\mathscr{H}_{K}\label{eq:fu2}
\end{gather}
So we get a pre-Hilbert space 
\[
\mathscr{M}_{2}\left(K\right):=\left\{ \text{signed measures \ensuremath{\mu} s.t. \ensuremath{\mu K\mu<\infty}}\right\} ,
\]
and 
\begin{equation}
T_{K}:\mathscr{M}_{2}\left(K\right)\longrightarrow\mathscr{H}_{K},
\end{equation}
where 
\begin{alignat}{3}
\mathscr{M}_{2}\left(K\right)\ni\mu & \quad & \xrightarrow{\quad T_{K}\quad} & \quad & T_{K}\mu\in\mathscr{H}_{K}\\
\mathscr{M}_{2}\left(K\right)\ni K^{-1}f &  & \xleftarrow[\quad T_{K}^{*}\quad]{} &  & f\in\mathscr{H}_{K}
\end{alignat}
So we have a well defined operator $T_{K}^{*}$, and we can use it
to make precise $K^{-1}$. So $K^{-1}$ gets a precise definition
via $T_{K}^{*}$. 

Recall we proved \eqref{eq:fu1}$\Longleftrightarrow$\eqref{eq:fu2},
\begin{equation}
\left\Vert T_{K}\mu\right\Vert _{\mathscr{H}_{K}}^{2}=\mu K\mu=\left\Vert \mu\right\Vert _{\mathscr{M}_{2}\left(K\right)}^{2}.
\end{equation}
So introduce 
\[
\left\langle \nu,\mu\right\rangle _{\mathscr{M}_{2}\left(K\right)}=\nu K\mu
\]
we have that 
\[
\mathscr{M}_{2}\left(K\right)\xrightarrow{\quad T_{K}\quad}m_{K}
\]
is isometric. 
\begin{prop}
Fix $K\in Pos\left(X\right)$. We have:
\[
\xymatrix{\mu\in\underset{\begin{matrix}\text{signed}\\
\text{measures}\\
\mu,\mu K\mu<\infty
\end{matrix}}{\underbrace{\mathscr{M}_{2}\left(K\right)}}\ar@/^{1.3pc}/[rr]^{T_{K}} &  & \underset{\text{RKHS}}{\mathscr{H}_{K}}\ar@/^{1.3pc}/[ll]^{T_{K}^{*}}}
\]
\begin{align}
T_{K}\mu & =\int\mu\left(dx\right)K\left(x,\cdot\right),\\
T_{K}^{*}f & =K^{-1}f,\label{eq:nm2}
\end{align}
where $K^{-1}$ is a Penrose-inverse (see e.g., \cite{MR4815055})
to $K$ where $K$ is interpreted as a kernel operator. 
\end{prop}

\begin{proof}
We must show the following identity for the respective inner products
on $f\in\mathscr{H}_{K}$, $\mu\in\mathscr{M}_{2}\left(K\right)$:
\begin{equation}
\left\langle T_{K}\mu,f\right\rangle _{\mathscr{H}_{K}}=\left\langle \mu,K^{-1}f\right\rangle _{\mathscr{M}_{2}\left(K\right)}.\label{eq:nm3}
\end{equation}
Using \eqref{eq:nm2}, we arrive at the following: 
\begin{align*}
\text{LHS}_{\left(\ref{eq:nm3}\right)} & =\left\langle \int\mu\left(dx\right)K\left(x,\cdot\right),f\right\rangle _{\mathscr{H}_{K}}\\
 & =\int\mu\left(dx\right)f\left(x\right)\\
 & =\int\mu\left(dx\right)K\left(K^{-1}f\right)\left(y\right)=\text{\ensuremath{\text{RHS}_{\left(\ref{eq:nm3}\right)}}. }
\end{align*}
Note that when $K^{-1}$ is acting via Penrose inverse on the function
$f$, the result $K^{-1}f$ is a signed measure, and that is the interpretation
used in the statement of the Proposition. 
\end{proof}
Of special significance to the above discussion are the following
citations \cite{MR4561157, MR4567343, MR4815055, MR3173967, MR4635411, MR3875593}.

\begin{flushleft}
\bibliographystyle{amsalpha}
\bibliography{ref}
\par\end{flushleft}
\end{document}